\documentclass[12pt,a4paper]{amsart}
\usepackage{mathrsfs}
\usepackage{amsfonts}
\usepackage{txfonts}

\usepackage{hyperref}
\usepackage{latexsym}
\usepackage{amssymb}

\newtheorem{theorem}{Theorem}[section]
\newtheorem{lemma}[theorem]{Lemma}
\newtheorem{corollary}[theorem]{Corollary}

\newtheorem{example}[theorem]{Example}

\theoremstyle{definition}
\newtheorem{definition}[theorem]{Definition}

\newtheorem{remark}[theorem]{Remark}

\numberwithin{equation}{section}

\begin{document}

\title[Maximal regularity for fractional difference equations]
{{\bf Maximal regularity for fractional difference equations with finite delay on UMD spaces}}

\author{Jichao Zhang,~~ Shangquan Bu$^\sharp$ }

\address{Jichao Zhang: School of Science, Hubei University of Technology, Wuhan 430068, China}
\email{156880717@qq.com (Jichao Zhang) \;}
\address{ Shangquan Bu:  Department of Mathematical Sciences, University of Tsinghua, Beijing 100084, China}
 \email{bushangquan@tsinghua.edu.cn (Shangquan Bu) \;}

\thanks{$^\sharp$ The corresponding author.}

 \thanks{$^1$ This work was supported by the NSF of China (Grant No. 12171266)}

\begin{abstract}  In this paper, we  study the $\ell^p$-maximal regularity for the fractional difference equation with finite delay:
$$\begin{cases}
 \Delta^{\alpha}u(n)=Au(n)+\gamma u(n-\lambda)+f(n), \ n\in \mathbb N_0, \lambda \in \mathbb N, \gamma \in \mathbb R; \\
u(i)=0,\ \ i=-\lambda, -\lambda+1,\cdots, 1,  2,
\end{cases}$$
where $A$ is a bounded linear operator  defined on  a Banach space $X$, $f:\mathbb N_0\rightarrow X$ is an $X$-valued sequence and $2<\alpha<3$.  We  introduce an operator theoretical method based on the  notion of $\alpha$-resolvent sequence of bounded linear operators, which gives an explicit representation of solution. Further, using Blunck's operator-valued Fourier multipliers theorems on $\ell^p(\mathbb{Z}; X)$, we completely characterize the  $\ell^p$-maximal regularity of solution when $1 < p < \infty$ and $X$ is a UMD space.
\end{abstract}



\maketitle

\section{Introduction}
The purpose of this paper is to study the existence and  uniqueness of solution and  the $\ell^p$-maximal regularity  for the fractional difference equation with finite delay:
\begin{equation}
 \Delta^{\alpha}u(n)=Au(n)+\gamma u(n-\lambda)+f(n), \ n\in \mathbb N_0, \lambda \in \mathbb N, \gamma \in \mathbb R
\end{equation}
with the initial conditions $u(i)=0$ when $i=-\lambda, -\lambda+1,\cdots, 1,  2$,
where $A$ is a bounded linear operator defined on a Banach space $X$, $f:\mathbb N_0\rightarrow X$ is  an $X$-valued sequence,  $2<\alpha<3$ and $1< p<\infty$. Here, we denote by $\mathbb{N}_0$ the set of all non negative integers,  $\Delta^{\alpha}$  denotes the discrete fractional operator of order $\alpha>0$ in the Riemann-Liouville sense (see Definition 2.1 in the second section).

The equation (1.1) has been the subject of research of  many authors due to its applications in many fields of sciences such as fractional nonlocal continuum mechanics and physics \cite{ta1, wu}. Much literature have been devoted to such problem \cite{gi, li5, le}. For instance, when $0<\alpha\leq1$, Lizama and Murillo-Arcila \cite{li5} studied the  $\ell^p$-maximal regularity  for  (1.1) with the initial conditions $u(i)=0$ when $i=-\lambda, -\lambda+1,\cdots, -1, 0$, and they  proved that when the underlying Banach space $X$ is a UMD space, $1<p< \infty$, $\sup_{n\in\mathbb N_0}\|S_\alpha(n)\|<\infty$  and $\{z^{1-\alpha}(z-1)^\alpha-\gamma z^{-\lambda}: |z|=1, z\not= 1\}\subset \rho (A)$, then (1.1) has  $\ell^p$-maximal regularity if and only if the sets $$\big\{z^{1-\alpha}(z-1)^\alpha[z^{1-\alpha}(z-1)^\alpha-\gamma z^{-\lambda}-A]^{-1}: |z|=1, z\neq 1\big\}$$  and $$\big\{ z^{-\lambda}[z^{1-\alpha}(z-1)^\alpha-\gamma z^{-\lambda}-A]^{-1}: |z|=1, z\neq 1\big\}$$ are $R$-bounded,
where $S_\alpha (n)\in B(X)$ is an appropriate  sequence defined the parameters $\alpha, \ \gamma,\ \lambda$ and $A$. Later, when $1<\alpha\leq2$, Leal, Lizama and Murillo-Arcila \cite{le}  further considered the $\ell^p$-maximal regularity  for  (1.1) with the initial  conditions $u(i)=0$ when $ i=-\lambda, -\lambda+1,\cdots, 0,  1$,  they  have shown that when  $X$ is a UMD space, $1<p< \infty$, $\sup_{n\in\mathbb N_0}\|S_\alpha(n)\|<\infty$  and  $\{z^{2-\alpha}(z-1)^\alpha-\gamma z^{-\lambda}: |z|=1, z\neq 1\} \subset \rho (A)$, then (1.1) has the $ \ell^p$-maximal regularity if and only if the sets $$\big\{z^{2-\alpha}(z-1)^\alpha[z^{2-\alpha}(z-1)^\alpha-\gamma z^{-\lambda}-A]^{-1}: |z|=1, z\neq 1\big\}$$  and $$\big\{ z^{-\lambda}[z^{2-\alpha}(z-1)^\alpha-\gamma z^{-\lambda}-A]^{-1}: |z|=1, z\neq 1\big\}$$ are $R$-bounded, where $S_\alpha (n)\in B(X)$ is also an appropriate  sequence defined $\alpha, \ \gamma,\ \lambda$ and $A$.  However, the validity of such characterization for the case $2<\alpha<3$ was left as an open problem. The aim of this paper is to give a positive answer to this problem.

In order to obtain our main results. We first introduce a similar sequence of bounded linear operators $S_\alpha(n)$ defined by  the parameters $\alpha, \ \gamma,\ \lambda$ and $A$,  called the $\alpha$-resolvent sequence $(S_\alpha(n))_{n\in \mathbb{N}_0}$, which will give an explicit representation of solution for (1.1). Precisely, let $2<\alpha<3$ and
 let $S_\alpha(-\lambda)= S_\alpha(-\lambda+1)=\cdots= S_\alpha(-1)=0$, $S_\alpha(0)= S_\alpha(1)= S_\alpha(2)= I$, and
\begin{align*}
S_\alpha(n+3) &- 2S_\alpha (n+2) + S_\alpha (n+1)= A(k^{\alpha-2}*S_\alpha)(n) +\gamma (k^{\alpha-2}*S_\alpha^\lambda)(n)\\
&+ k^{\alpha-2}(n+3)I+ (1-\alpha)k^{\alpha-2}(n+2)I+ \frac{(\alpha-1)(\alpha-2)}{2}k^{\alpha-2}(n+1)I
\end{align*}
 when $n\in \mathbb N_0$, $\lambda \in \mathbb N$ and $\gamma \in \mathbb R$, where $S_\alpha^\lambda(n)=S_\alpha(n-\lambda)$ and $ k^{\alpha -2}$ is defined by (2.4) below.
We show that  when $f: \mathbb N_0\to X$ is given, there exists a unique solution $u: \mathbb N_0\to X$ of (1.1) which is given by the formula
\begin{align*}
u(n)=(h_\alpha* S_\alpha*f)(n-3)
\end{align*}
when $n\geq3$. Here the function $h_\alpha: \mathbb N_0\rightarrow \mathbb R$ is defined by $h_\alpha(0)=1, h_\alpha(1)=\alpha-1$, $h_\alpha(2)=\alpha(\alpha-1)/2$ and
\begin{equation}\nonumber
h_\alpha(n+3)+(1-\alpha) h_\alpha(n+2)+ (\alpha-1)(\alpha-2)h_\alpha(n+1)/2=0
\end{equation}
when $n\in \mathbb{N}_0$.

We show that when the underlying Banach space $X$ is a UMD space, $1<p< \infty$, $\sup_{n\in\mathbb N_0}\|S_\alpha(n)\|<\infty$ and  $\{z^{3-\alpha}(z-1)^\alpha-\gamma z^{-\lambda}: |z|=1, z\not= \pm 1\}\subset \rho (A)$, then
(1.1) has the $\ell^p$-maximal regularity if and only if
the sets
\begin{align}
\big\{z^{3-\alpha}(z-1)^\alpha[z^{3-\alpha}(z-1)^\alpha-\gamma z^{-\lambda}-A]^{-1}: |z|=1, z\neq \pm1\big\}
\end{align}
and
\begin{align}
\big\{z^{-\lambda}[z^{3-\alpha}(z-1)^\alpha-\gamma z^{-\lambda}-A]^{-1}: |z|=1, z\neq \pm1\big\}
\end{align}
are $R$-bounded.

It is clear that the $R$-boundedness of the sets (1.2) and (1.3) do not depend on the space parameter $p$. Thus  when the underlying Banach space $X$ is a UMD space, $\sup_{n\in\mathbb N_0}\|S_\alpha(n)\|<\infty$, $\{z^{3-\alpha}(z-1)^\alpha-\gamma z^{-\lambda}: |z|=1, z\not= \pm1\}\subset \rho (A)$, if (1.1) has the $\ell^p$-maximal regularity for some $1<p< \infty$, then (1.1) has  the $\ell^p$-maximal regularity for all $1<p< \infty$.  Since every norm bounded subset of $B(X)$ is actually $R$-bounded when $X$ is a Hilbert space, we deduce that when the underlying Banach space $X$ is a Hilbert space,  $1<p< \infty$, $\sup_{n\in\mathbb N_0}\|S_\alpha(n)\|<\infty$  and $\{z^{3-\alpha}(z-1)^\alpha-\gamma z^{-\lambda}: |z|=1, z\not=\pm 1\}\subset \rho (A)$, then (1.1) has the $\ell^p$-maximal regularity if and only if the sets (1.2) and (1.3) are norm bounded.

We notice that the fractional difference equation (1.1) in the case $2<\alpha<3$ and $\gamma =0$ was previously studied in \cite{zh}. Our results
recover the explicit formula of solution and the characterization of $\ell^p$-maximal regularity for (1.1) obtained in \cite{zh}. The main tool in this paper is the operator-valued Fourier multipliers theorems on $\ell^p(\mathbb{Z}; X)$  obtained by Blunck \cite{bl}, we will transform the $\ell^p$-maximal regularity of  (1.1) to an operator-valued Fourier multiplier problem on $\ell^p(\mathbb{Z}; X)$.

We notice that the research in this paper was motivated by recent studies on discrete mathematical models proved that they can serve as a new microstructural basis for fractional nonlocal continuum mechanics and physics \cite{ta, ta1}. Fractional difference equations can be also used to formulate adequate models in nanomechanics  \cite{ta1, wu} and therefore further studies in this class of fractional difference equations deserve to be investigated. Our contribution in this paper provides a new qualitative advance in this line of research, that incorporates tools from operator theory and allows the analysis of
maximal regularity for a very general but still simple model. The study of more complex dynamical systems that include unbounded operators is still an open problem. This task will be the objective of forthcoming works.

This paper is organized as follows: in section 2, we recall some basic concepts and results in the existing literature related to fractional difference operators, UMD spaces, $R$-boundedness,   the discrete time Fourier transform and Blunck's Fourier multipliers theorems that will be later used; in section 3, we study the $\alpha$-resolvent sequence defined by the parameters $\alpha, \ \gamma,\ \lambda$ and $A$, and give an explicit representation of solution for (1.1); in the last section, we give a characterization of the $\ell^p$-maximal regularity for equation (1.1) when $1 < p <\infty$ and $X$ is a UMD space.

\section{Preliminaries}
In this section, we   briefly recall some necessary notions and results  related to fractional difference operators,  UMD spaces, $R$-boundedness, the discrete time Fourier transform and Blunck's Fourier multipliers theorems, which will be used in the sequel.

Let $X$ be a Banach space. We denote by $S(\mathbb N_0; X)$ the vector space consisting of all  vector-valued sequences $u:\mathbb N_0\rightarrow X$. Similarly we denote by $S(\mathbb Z; X)$ the vector space  consisting of all vector-valued sequences $u:\mathbb Z\rightarrow X$.
The forward Euler operator $\Delta: S(\mathbb N_0; X)\rightarrow S(\mathbb N_0; X)$ is defined as follows:
\begin{align*}
\Delta u(n)= u(n+1)-u(n)
\end{align*}
when $ n\in \mathbb N_0$.
For every $m \in \mathbb N$, we define recursively the m-th order forward difference operator $\Delta^m: S(\mathbb N_0;X)\rightarrow S(\mathbb N_0;X)$ by
$\Delta^m=\Delta^{m-1} \Delta$.

Recall that the finite convolution $*$ between two sequences $f\in S(\mathbb N_0; \mathbb{C})$ and $g\in S(\mathbb N_0;X)$ is defined by
\begin{equation}
(f*g)(n):=\sum_{j=0}^{n}f(n-j)g(j)
\end{equation}
when $n\in \mathbb N_0$. It is easy to verify that if $h\in S(\mathbb N_0; \mathbb{C})$, then
 \begin{equation}
 ((f*h)*g)(n) = (f*(h*g))(n) = \sum_{i+j+k =n} f(i)h(j)g(k)
 \end{equation}
 when $n\in\mathbb{N}_0$. Hence we will use the expression $f*g*h$ without any confusion.

The following definition of fractional sum, due to the previous works \cite{ab, at1},  was formally  presented  by Lizama in \cite{li}.
Let $0 <\beta \leq 1$ and $u\in S(\mathbb N_0; X)$ be given. The fractional sum of $u$ of order $\beta$ is defined by
\begin{equation}
\Delta^{-\beta}u(n):= (k^\beta * u)(n) = \sum_{j=0}^n k^{\beta}(n-j)u(j)
\end{equation}
when $n\in \mathbb N_0$,
where
\begin{align}
k^{\beta}(j):=\frac{\Gamma(\beta+j)}{\Gamma(\beta)\Gamma(j+1)}
\end{align}
when $j\in \mathbb N_0$, where $\Gamma$ is the Gamma function. It is easy to verify  that this sequence $(k^{\beta}(j))_{n\in \mathbb{N}_0}$ satisfies the following semigroup property
\begin{align}
(k^{\beta} * k^{\alpha})(n)= k^{\beta+\alpha}(n)
\end{align} when $n\in \mathbb N_0, \beta, \alpha\in \mathbb{C}$.

 The next definition corresponds to an analogous version of fractional derivative in the sense of Riemann-Liouville, see \cite{ag, mi}.

\begin{definition} Let $\alpha>0,\ \alpha \notin \mathbb{N}$ and $u\in S(\mathbb N_0; X)$ be given. The fractional difference operator of $u$ of order $\alpha$ is defined by
\begin{equation}
\Delta^{\alpha}u(n):=\Delta^m\Delta^{-(m-\alpha)}u(n)
\end{equation}
when $n\in \mathbb N_0$,
 where $m\in \mathbb{N}$ is the unique integer $m$ satisfying $m-1<\alpha < m$. For more study of fractional difference operators, we refer the readers to \cite{go, li1, li2}
\end{definition}

Let $u\in S(\mathbb Z; X)$ be given,  the discrete time Fourier transform of $u$ is given by
\begin{equation*}
\hat{u}(z) := \sum_{j=-\infty}^{\infty}z^{-j}u(j)
\end{equation*}
for all $\vert z\vert =1$, whenever it exists.  We notice that the Fourier transform of $ u$ is sometimes also denoted by $\mathcal{F} (u)$. It is clear that when $f\in S(\mathbb N_0; \mathbb{C})$ and $g\in S(\mathbb N_0;X)$, then
\begin{equation}
\widehat{(f*g)}(z) = \hat f(z)\hat g(z)
\end{equation}
 when both sides are well defined for all $\vert z\vert =1$.

Let $k^\alpha$ be defined by (2.4) for $0 < \alpha \leq 1$. It follows  that the Fourier transform of $k^\alpha$ is given by
\begin{align}\
\hat k^\alpha (z) = \frac{z^\alpha}{(z-1)^\alpha}
\end{align}
when $\vert z\vert =1$ and $z\not= 1$.

 We say that a Banach space $X$  is a UMD (Unconditionality of Martingale Differences) space if for all $1 < p<\infty$, there exists a constant $C >0$ (depending only on $p$ and $X$) such that for any martingale $(g_n)_{n\geq0}\subset L^p(\Omega,\Sigma,\mu; X)$ and all scalars $|\varepsilon_n|=1, n=1, 2\cdots, N$, the following inequality holds:
 \begin{equation*}
 \Big\|g_0+\sum_{n=1}^{N}\varepsilon_n(g_n-g_{n-1})\Big\|_{L^p(\Omega,\Sigma,\mu; X)}\leq C
 \big\|g_N\big\|_{L^p(\Omega,\Sigma,\mu; X)}.
 \end{equation*}
 It is well known that $L_p$ spaces, Schatten class $S_p$ and Sobolev spaces $W^{m,p}$ are UMD spaces when $1 < p < \infty$.
UMD spaces have played a very important part in  vector-valued harmonic analysis and probability theory, see \cite{bu}.

  Let $X$ be a Banach space, we denote by $B(X)$ the space of all bounded linear operators on $X$.
 Let $r_j$ be the $j$-th Rademacher function defined on $[0,1]$ by $r_j(t)=\text{sgn(sin}(2^{j}t))$ whenever $j\geq1$.

\begin{definition} Let $X$  be a  Banach spaces.  A set $W\subset B(X)$  is said to be Rademacher bounded ($R$-bounded, shortly) \cite{hy}, if there exists $C\geq0$ such that
\begin{align*}
\Big\|\sum_{j=1}^{n}r_j T_jx_j\Big\|_{L^1([0,1];X)}\leq C\Big\|\sum_{j=1}^{n}r_j x_j\Big\|_{L^1([0,1];X)}
 \end{align*}
 for all $T_1, T_2,\cdots, T_n\in W$, $x_1, x_2,\cdots, x_n\in X$ and $n\in \mathbb N.$
\end{definition}

\begin{remark}
 \rm It is clear that when $\mathrm{W_1}, \mathrm{W_2}\subset B(X)$ are $R$-bounded sets, then the sets
 $$\mathrm{W_1 W_2 :=\big\{ST: S\in W_1, T\in W_2\big\}},\ \mathrm{W_1+W_2}:= \big\{S+T: S\in W_1, T\in W_2\big\}$$
  are still $R$-bounded. It is easy to see that if $\mathrm{W}$ is a bounded subset of the complex plane $\mathbb C$, then the set $\{\mu I: \mu\in \mathrm{W}\}$ is also $R$-bounded, where $I$ stands for the identity operator on $X$. This follows easily from the Kahane's contraction principle \cite{li0}.
\end{remark}

In what follows we denote by $\mathbb{T}=(-\pi, 0)\cup(0, \pi)$. Let $X$ be a Banach space and let  $G:\mathbb{T}\rightarrow B(X)$ be bounded and measurable. Let $f\in S(\mathbb{Z}; X)$ with finite support, i.e., the set $\{n\in \mathbb{Z}: f(n) \not=0\}$ is finite. Then the function $t\to G(t)(\mathcal{F} f)(e^{it})$ defined on $t\in \mathbb{T}$ is bounded and measurable. Thus its inverse Fourier transform
\begin{equation*}
\mathcal{F}^{-1}\big[G(\cdot) (\mathcal{F} f)(e^{i\cdot})\big] (n) = \frac{1}{2\pi} \int_0^{2\pi} G(t)(\mathcal{F} f)(e^{it}) e^{int} dt
\end{equation*}
makes sense for all $n\in\mathbb{Z}$. Let $1\leq p < \infty$ be given. We say that $G$ is an $\ell^p$-Fourier multiplier, if there exists a constant $C> 0$ independent  from
$f$, such that
\begin{equation*}
\Big(\sum_{n\in\mathbb{Z}} \big\Vert \mathcal{F}^{-1}\big[G(\cdot) (\mathcal{F} f)(e^{i\cdot})\big] (n) \big\Vert^p\Big)^{1/p} \leq C \Big(\sum_{n\in\mathbb{Z}} \Vert f(n)\Vert^p\Big)^{1/p}
\end{equation*}
for all $f\in S(\mathbb{Z}; X)$ with finite support. In this case, there exists a unique
bounded linear operator
$T_G \in B(\ell^p(\mathbb{Z}; X))$, such that
\begin{equation}
G(t) (\mathcal{F} f)(e^{it}) = \mathcal{F} (T_G f)(e^{it})
 \end{equation}
 when $t\in \mathbb{T}$ and  $f\in S(\mathbb{Z}; X)$ with finite support. Here we used the easy fact that the set of all elements $f\in S(\mathbb{Z}; X)$ with finite support is dense in $\ell^p(\mathbb{Z}; X)$.

We finish this section with the following $\ell^p$-Fourier multipliers established by Blunck \cite{bl}, which will be fundamental in our investigation.

\begin{theorem} Let $1 < p < \infty$ and let $X$ be a UMD space. Assume that  $G:\mathbb{T}\rightarrow B(X)$ is differentiable and the sets
\begin{equation*}
\big\{G(t): t\in \mathbb{T}\big\}~~and \ \ \big\{(e^{it}-1)(e^{it}+1)G'(t): t\in \mathbb{T}\big\}
\end{equation*}
 are $R$-bounded. Then $G$ is an $\ell^p$-Fourier multiplier.
\end{theorem}

\begin{theorem}
Let $1 \leq p < \infty$ and let $X$ be a Banach space.  Let $G: \mathbb{T}\rightarrow B(X)$ be continuous and bounded. Assume  that $G$ is an $\ell^p$-Fourier multiplier. Then the set $\{G(t): t\in \mathbb{T}\}$ is $R$-bounded.
\end{theorem}

\section{The Representation of Solutions}
In this section, we will introduce  a concept of  $\alpha$-resolvent sequence of bounded operators,  which will give  an explicit representation of the solution for the fractional difference equation (1.1) when $2 < \alpha<3$. We observe that the equation (1.1) in the case $0< \alpha \leq1$ with the initial conditions $u(i)=0$ when $ i=-\lambda, -\lambda+1,\cdots, -1, 0$ was previously studied by Lizama and Murillo-Arcila \cite{li5}. Later, Leal, Lizama and Murillo-Arcila \cite{le} further studied the equation (1.1) in the case $1< \alpha \leq2$ with the initial conditions  $u(i)=0$ when $i=-\lambda, -\lambda+1,\cdots, 0,  1$.

\subsection{{\rm \textbf{The $\alpha$-Resolvent Sequence with Finite Delay }}}

\vskip .3cm
We first introduce an similar $\alpha$-resolvent sequence $S_\alpha (n)$ when $n\geq -\lambda$  used in  \cite{li5, le}.

\begin{definition}
Let $A$ be a bounded linear operator defined on a  Banach space $X$, and let $\gamma \in \mathbb R$, $\lambda\in\mathbb N$ and $2 < \alpha < 3$. We define the $\alpha$-resolvent sequence $S_\alpha (n)$ when $n\geq -\lambda$ generated by $A, \alpha, \gamma$ and $\lambda$    by

(i) $S_\alpha(-i)= 0$ when $ i= 1, \cdots, \lambda$;;

(ii) $S_\alpha(0)= S_\alpha(1)= S_\alpha(2)= I$;

(iii) $S_\alpha(n+3) - 2S_\alpha (n+2) + S_\alpha (n+1)= A(k^{\alpha-2}*S_\alpha)(n) +\gamma ( k^{\alpha-2}*S_\alpha^\lambda)(n)+ k^{\alpha-2}(n+3)I + (1-\alpha)k^{\alpha-2}(n+2)I + \frac{(\alpha-1)(\alpha-2)}{2}k^{\alpha-2}(n+1)I$ when $ n\in\mathbb N_0$, where $S_\alpha^\lambda(n)$ is defined by $S_\alpha^\lambda(n)=S_\alpha(n-\lambda)$.
\end{definition}

\begin{remark}  Note that the above Definition 3.1 corresponds to the definition of $\alpha$-resolvent sequence defined in \cite{zh} when $\gamma=0$. We need to define $S_\alpha (-1) = \cdots = S_\alpha (-\lambda) = 0$ as $S_\alpha^\lambda(n)$ is defined by $ S_\alpha(n-\lambda)$ when $n\in \mathbb N_0$.
\end{remark}

\begin{remark}  Let  $\gamma\in\mathbb R$, $\lambda\in\mathbb N$ and $2<\alpha<3$. Assume that $z^{3-\alpha}(z-1)^\alpha-\gamma z^{-\lambda}\in\rho (A)$ for all  $|z|=1,\ z\not=  1$. Then the Fourier transform of $S_\alpha$ is given by
 \begin{equation}
\widehat{S}_\alpha(z)=\Big[z^3+(1-\alpha)z^2+\frac{(\alpha-1)(\alpha-2)}{2}z\Big]\big[z^{3-\alpha}(z-1)^\alpha-\gamma z^{-\lambda}-A\big]^{-1}
\end{equation}
when $|z|=1,\ z\not= 1$. Indeed, taking the  Fourier transform  on both sides of (iii) in Definition 3.1 and using (2.4), (2.7) and (i), (ii) in Definition 3.1, we obtain
 \begin{equation*}
 \begin{aligned}
z^3\widehat{S}_\alpha(z)&-z^3-z^2-z-2\big[z^2\widehat{S}_\alpha(z)-z^2-z\big]+z\widehat{S}_\alpha(z)-z=
A\widehat{k}^{\alpha-2}(z)\widehat{S}_\alpha(z)\\
&+\gamma z^{-\lambda}\widehat{k}^{\alpha-2}(z)\hat{S}_\alpha^\lambda(z)+z^3\widehat{k}^{\alpha-2}(z)-z^3-z^2(\alpha-2)-\frac{(\alpha-1)(\alpha-2)}{2}z\\
&+(1-\alpha)\big[z^2\widehat{k}^{\alpha-2}(z)-z^2-z(\alpha-2)\big]+\frac{(\alpha-1)(\alpha-2)}{2}\big(z\widehat{k}^{\alpha-2}(z)-z\big)\\
\end{aligned}
\end{equation*}
when $|z|=1,\ z\not=  1$. Here we have used the fact that the function $\widehat{k}^{\alpha-2}(z)$ is only well-defined when  $|z|=1,\ z\not=  1$
by (2.8).
It follows that
 \begin{equation}\nonumber
\big[z^3-2z^2+z-\gamma z^{-\lambda}\widehat{k}^{\alpha-2}(z)-A\widehat{k}^{\alpha-2}(z)\big]\widehat{S}_\alpha(z)=\Big[z^3+(1-\alpha)z^2+\frac{(\alpha-1)
(\alpha-2)}{2}z\Big]\widehat{k}^{\alpha-2}(z)
\end{equation}
when $|z|=1,\ z\not=  1$. Thus, by  (2.8), we have
 \begin{equation}\nonumber
\widehat{S}_\alpha(z)=\Big[z^3+(1-\alpha)z^2+\frac{(\alpha-1)(\alpha-2)}{2}z\Big]\big[z^{3-\alpha}(z-1)^\alpha-\gamma z^{-\lambda}-A\big]^{-1}
\end{equation}
when $|z|=1,\ z\not=  1$.
\end{remark}

We will need the following Lemma proved in \cite{zh}.

\begin{lemma} {\rm (\cite[Lemma 3.1]{zh})}
Let $2<\alpha<3$, $b:\mathbb N_0\rightarrow \mathbb C$ and $P:\mathbb N_0\rightarrow X$, where $X$ is a Banach space. Then
\begin{align*}
% \begin{aligned}
\Delta^{\alpha}(b*P)(n)&=(b*\Delta^{\alpha}P)(n)+b(n+3)P(0)+b(n+2)\big[P(1)-\alpha P(0)\big]\nonumber\\
&+b(n+1)\Big[P(2)-\alpha P(1)+\frac{\alpha(\alpha -1)}{2}P(0)\Big]
%\end{aligned}
\end{align*}
when $n\in \mathbb N_0$.
\end{lemma}

In order to prove our main result of this section, we need the following lemmas.

\begin{lemma}
Let $X$ be a Banach space, $A\in B(X)$, $2<\alpha<3$, $\gamma\in\mathbb R$, $\lambda\in\mathbb N$ and let $(S_\alpha (n))_{n\geq -\lambda}$ be the $\alpha$-resolvent sequence given by Definition 3.1. Then
\begin{equation*}
\Delta^{\alpha}S_\alpha(n)=AS_\alpha (n)+\gamma S_\alpha^\lambda(n)
\end{equation*}
when $0\leq n\leq 2$.
\end{lemma}

\begin{proof}
By using (2.1), (2.3) and (2.6), we have
\begin{align}
\Delta^{\alpha}S_\alpha(0)&=\Delta^3\Delta^{-(3-\alpha)}S_\alpha(0)\\
&=\Delta^{-(3-\alpha)}S_\alpha(3)-3\Delta^{-(3-\alpha)}S_\alpha(2)\nonumber\\
&\ \ \ \ +3\Delta^{-(3-\alpha)}S_\alpha(1)-\Delta^{-(3-\alpha)}S_\alpha(0)\nonumber\\
&=\sum_{j=0}^{3}k^{3-\alpha}(3-j)S_\alpha(j)-3\sum_{j=0}^{2}k^{3-\alpha}(2-j)S_\alpha(j)\nonumber\\
&\ \ \ \ +3\sum_{j=0}^{1}k^{3-\alpha}(1-j)S_\alpha(j)-k^{3-\alpha}(0)S_\alpha(0)\nonumber\\
&=S_\alpha(3)+k^{3-\alpha}(1)I-2k^{3-\alpha}(2)I+k^{3-\alpha}(3)I-I.\nonumber
\end{align}
It follows from Definition 3.1  that
\begin{align}
S_\alpha(3)=&I+A( k^{\alpha-2}*S_\alpha)(0)+\gamma (k^{\alpha-2} *S_\alpha^\lambda )(0)\\ &+k^{\alpha-2}(3)I+(1-\alpha)k^{\alpha-2}(2)I+\frac{(\alpha-1)(\alpha-2)}{2}k^{\alpha-2}(1)I.\nonumber
\end{align}

Thus, (3.2)and (3.3) together imply that
\begin{align*}\nonumber
\Delta^{\alpha}S_\alpha(0)&=I+A(k^{\alpha-2}*S_\alpha)(0)+\gamma (k^{\alpha-2}*S_\alpha^\lambda)(0)+k^{\alpha-2}(3)I+(1-\alpha)k^{\alpha-2}(2)I\nonumber\\
&\ \ \ \ +\frac{(\alpha-1)(\alpha-2)}{2}k^{\alpha-2}(1)I+k^{(3-\alpha)}(1)I-2k^{(3-\alpha)}(2)I+k^{(3-\alpha)}(3)I-I\nonumber\\
&=A(k^{\alpha-2}*S_\alpha)(0)+\gamma ( k^{\alpha-2}* S_\alpha^\lambda)(0)\nonumber\\
&=A(k^{\alpha-2}*S_\alpha)(0)+\gamma S_\alpha^\lambda(0)\nonumber\\
&=AS_\alpha(0)+\gamma S_\alpha^\lambda(0)\nonumber.
\end{align*}

 From Definition 3.1 again, we have
\begin{align}
S_\alpha (4) = &2S_\alpha (3) - I + A(k^{\alpha -2}*S_\alpha)(1)+\gamma (k^{\alpha-2}*S_\alpha^\lambda)(1)\\
 & + k^{\alpha -2} (4)I+ (1-\alpha)k^{\alpha -2} (3)I + \frac{(\alpha-1)(\alpha-2)}{2}k^{\alpha -2} (2)I\nonumber
\end{align}
and
\begin{align}
S_\alpha (5) = &2S_\alpha (4) - S_\alpha (3) + A (k^{\alpha -2}*S_\alpha)(2) +\gamma (k^{\alpha-2}* S_\alpha^\lambda)(2)\\
 &+ k^{\alpha-2}(5)I+ (1-\alpha)k^{\alpha-2}(4)I + \frac{(\alpha-1)(\alpha-2)}{2}k^{\alpha-2}(3)I.\nonumber
\end{align}

Again by  (2.1), (2.3) and (2.6), we obtain
\begin{align}
\Delta^{\alpha}S_\alpha(1)&=\Delta^3\Delta^{-(3-\alpha)}S_\alpha(1)\\
&=\Delta^{-(3-\alpha)}S_\alpha(4)-3\Delta^{-(3-\alpha)}S_\alpha(3)\nonumber\\
&\ \ \ \ +3\Delta^{-(3-\alpha)}S_\alpha(2)-\Delta^{-(3-\alpha)}S_\alpha(1)\nonumber\\
&=\sum_{j=0}^{4}k^{3-\alpha}(4-j)S_\alpha(j)-3\sum_{j=0}^{3}k^{3-\alpha}(3-j)S_\alpha(j)\nonumber\\
&\ \ \ \ +3\sum_{j=0}^{2}k^{3-\alpha}(2-j)S_\alpha(j)-\sum_{j=0}^{1}k^{3-\alpha}(1-j)S_\alpha(j)\nonumber\\
&=S_\alpha(4)-\alpha S_\alpha (3)+ k^{3-\alpha}(2)I-2k^{3-\alpha}(3)I\nonumber\\
&\ \ \ \ +k^{3-\alpha}(4)I+(\alpha -1)I.\nonumber
\end{align}
Therefore, by using (3.3), (3.4) and (3.6) we deduce that
\begin{align*}
\Delta ^\alpha S_\alpha (1) &= (2-\alpha)S_\alpha (3)- I + A(k^{\alpha -2}*S_\alpha)(1) +\gamma ( k^{\alpha-2}*S_\alpha^\lambda)(1)I\\
&\ \ \ \ + k^{\alpha -2} (4) + (1-\alpha)k^{\alpha -2}(3)I + \frac{(\alpha-1)(\alpha-2)}{2}k^{\alpha-2}(2)I\\
&\ \ \ \ + k^{3-\alpha}(2)I-2k^{3-\alpha}(3)I+k^{3-\alpha}(4)I+(\alpha -1)I\\
&=AS_\alpha (1) + k^{\alpha-2}(4)I + (3-2\alpha) k^{\alpha-2}(3)I +\frac{3(\alpha -1)(\alpha -2)}{2} k^{\alpha-2}(2)I\\
&\ \ \ \  - \frac{(\alpha -1)(\alpha -2)^2}{2}k^{\alpha-2}(1)I
 + k^{3-\alpha}(2)I-2k^{3-\alpha}(3)I\\
 & \ \  \ \ +k^{3-\alpha}(4)I
+(2-\alpha)\gamma (k^{\alpha-2}* S_\alpha^\lambda)(0)+\gamma ( k^{\alpha-2}*S_\alpha^\lambda )(1)\\
&=AS_\alpha (1)+\gamma S_\alpha^\lambda(1).
\end{align*}

Using a similar argument in (3.6), we have
\begin{align}
\Delta^{\alpha}S_\alpha(2)=&S_\alpha(5)-\alpha S_\alpha (4)+ \frac{\alpha (\alpha -1)}{2}S_\alpha (3) + k^{3-\alpha}(5)I -2k^{3-\alpha}(4)I\\
& + k^{3-\alpha}(3)I-k^{3-\alpha}(2)I + 2k^{3-\alpha}(1)I - k^{\alpha-2}(0)I.\nonumber
\end{align}
Therefore, applying (3.3), (3.4), (3.5) and (3.7), we deduce that
\begin{align*}
\Delta ^\alpha S_\alpha (2) &= AS_\alpha(2) + \frac{\alpha (\alpha -3)}{2}I + k^{\alpha-2}(5)I + (3-2\alpha) k^{\alpha-2}(4)I\\
&\ \ \ \ +(\alpha -2)(2\alpha -3) k^{\alpha-2}(3)I +\frac{(2-\alpha)(1-\alpha)(5-2\alpha)}{2} k^{\alpha-2}(2)I \\
&\ \ \ \ -\frac{(\alpha -1)(3-\alpha)(\alpha -2)^2}{4}k^{\alpha-2}(1)I+ k^{3-\alpha}(5)I -2k^{3-\alpha}(4)I\\
&\ \ \ \  + k^{3-\alpha}(3)I -k^{3-\alpha}(2)I + 2k^{3-\alpha}(1)I+ \frac{(\alpha-2)(\alpha-3)}{2} \gamma S_\alpha^\lambda(0) \\
&\ \ \ \ +(2-\alpha)\gamma (k^{\alpha-2}* S_\alpha^\lambda)(1)+\gamma (k^{\alpha-2}*S_\alpha^\lambda )(2)\\
& =AS_\alpha(2)+\gamma S_\alpha^\lambda(2).
\end{align*}
The proof is completed.
\end{proof}

\begin{lemma}
Let $X$ be a Banach space, $A\in B(X)$, $2<\alpha<3$, $\gamma\in\mathbb R$, $\lambda\in\mathbb N$ and let $(S_\alpha (n))_{n\geq -\lambda}$ be the $\alpha$-resolvent sequence given by Definition 3.1. Then
\begin{equation*}
\Delta^{\alpha}S_\alpha(n)= AS_\alpha (n)+\gamma S_\alpha^\lambda(n)
\end{equation*}
for all $n\in \mathbb{N}_0$.
\end{lemma}

\begin{proof}
First note that
\begin{equation}
\Delta^{\alpha}k^{\alpha-2}(n)=\Delta^3\Delta^{-(3-\alpha)}k^{\alpha-2}(n)=\Delta^3(k^{(3-\alpha)}*k^{\alpha-2})(n)=\Delta^3k^1(n)=0
\end{equation}
when $n\in \mathbb N_0$ as $k^1(n) = 1$ for all $n\in \mathbb N_0$. By (iii)  in Definition 3.1, we have
\begin{align*}
\sum_{j=0}^n k^{3-\alpha} (j) S&_\alpha(n+3 -j) -2 \sum_{j=0}^n k^{3-\alpha} (j) S_\alpha (n+2 -j) + \sum_{j=0}^n k^{3-\alpha} (j) S_\alpha (n+1 -j)\nonumber\\
&= A \sum_{j=0}^n k^{3-\alpha} (j) (k^{\alpha -2}* S_\alpha) (n -j) +\gamma \sum_{j=0}^n k^{3-\alpha}(j)( k^{\alpha-2}*S_\alpha^\lambda)(n-j)\nonumber\\
&\ \ \ \  + \sum_{j=0}^nk^{3-\alpha}(j)k^{\alpha -2} (n +3 -j)I+(1-\alpha)\sum_{j=0}^nk^{3-\alpha}(j) k^{\alpha -2} (n+2 -j)I \nonumber\\
&\ \ \ \  + \frac{(\alpha -1)(\alpha -2)}{2}\sum_{j=0}^n k^{3-\alpha}(j) k^{\alpha -2} (n +1-j)I
\end{align*}
when $n\in\mathbb{N}_0$.
This implies that
\begin{align*}
\Delta^{-(3-\alpha)}& S_\alpha (n+3) -2 \Delta^{-(3-\alpha)} S_\alpha (n+2) + \Delta^{-(3-\alpha)} S_\alpha (n+1)\\
&\ \ \ \ \ -k^{3-\alpha} (n+1) S_\alpha (2) - k^{3-\alpha} (n+2)S_\alpha (1)-S_\alpha (0) k^{3-\alpha} (n+3)\\
&\ \ \ \ \ + 2 k^{3-\alpha} (n+1)S_\alpha (1)+ 2 k^{3-\alpha} (n+2) S_\alpha (0) - k^{3-\alpha} (n+1) S_\alpha (0)\\
&= A\Delta ^{-(3-\alpha)} (k^{\alpha -2}* S_\alpha)(n)+\gamma\Delta ^{-(3-\alpha)} (k^{\alpha-2}* S_\alpha^\lambda)(n) + \Delta ^{-(3-\alpha)} k^{\alpha -2} (n+3)I\\
&\ \ \ \  +(1-\alpha)\Delta ^{-(3-\alpha)} k^{\alpha -2} (n+2)I +  \frac{(\alpha -1)(\alpha -2)}{2}\Delta ^{-(3-\alpha)} k^{\alpha -2} (n+1)I\\
&\ \ \ \ - k^{3-\alpha} (n+3)k^{\alpha-2}(0)I
 - k^{3-\alpha} (n+2)k^{\alpha -2}(1)I- k^{3-\alpha} (n+1)k^{\alpha-2}(2)I\\
&\ \ \ \ -(1-\alpha)k^{3-\alpha} (n+2)k^{\alpha-2}(0)I
-(1-\alpha)k^{3-\alpha} (n+1)k^{\alpha-2}(1)I\\
&\ \ \ \  - \frac{(\alpha -1)(\alpha -2)}{2} k^{3-\alpha} (n+1)k ^{\alpha -2}(0)I
\end{align*}
when $n\in\mathbb{N}_0$. It follows that
\begin{align*}
\Delta^{-(3-\alpha)}& S_\alpha (n+3) -2 \Delta^{-(3-\alpha)} S_\alpha (n+2) + \Delta^{-(3-\alpha)} S_\alpha (n+1)\\
 &= A\Delta ^{-(3-\alpha)} ( k^{\alpha -2}*S_\alpha)(n)+\gamma\Delta ^{-(3-\alpha)} ( k^{\alpha-2}*S_\alpha^\lambda)(n) + \Delta ^{-(3-\alpha)} k^{\alpha -2} (n+3)\\
&\ \ \ \ + \frac{(\alpha -1)(\alpha -2)}{2}\Delta ^{-(3-\alpha)} k^{\alpha -2} (n+1) +(1-\alpha)\Delta ^{-(3-\alpha)} k^{\alpha -2} (n+2)
\end{align*}
when $n\in\mathbb{N}_0$. Therefore,
\begin{align}
\Delta ^\alpha S_\alpha (n+3) -2 \Delta ^\alpha S_\alpha (n+2) + \Delta ^\alpha S_\alpha (n+1)=&A\Delta ^{\alpha} ( k^{\alpha -2}*S_\alpha )(n)\\
&+\gamma\Delta^\alpha (k^{\alpha-2}* S_\alpha^\lambda)(n).\nonumber
\end{align}

by (3.8) when $n\in\mathbb{N}_0$.

Applying Lemma 3.4 and (3.8), we obtain
\begin{align}
\Delta^{\alpha}(k^{\alpha-2}* S_\alpha)(n)=&(\Delta^{\alpha}k^{\alpha-2}* S_\alpha)(n)+k^{\alpha-2}(0)S_\alpha(n+3)\\
&\ \ \ \ +[k^{\alpha-2}(1)-\alpha k^{\alpha-2}(0)]S_\alpha(n+2)+\Big[k^{\alpha-2}(2)\nonumber\\
&\ \ \ \ -\alpha k^{\alpha-2}(1)+ \frac{\alpha (\alpha -1)}{2}k^{\alpha-2}(0)\Big]S_\alpha(n+1)\nonumber\\
&=S_\alpha(n+3)-2S_\alpha(n+2)+S_\alpha(n+1)\nonumber
\end{align}
 and
 \begin{align}
\Delta^{\alpha}(k^{\alpha-2}* S_\alpha^\lambda)(n)&=(\Delta^{\alpha}k^{\alpha-2}* S_\alpha^\lambda)(n)+k^{\alpha-2}(0)S_\alpha^\lambda(n+3)\\
&\ \ \ \ +[k^{\alpha-2}(1)-\alpha k^{\alpha-2}(0)]S_\alpha^\lambda(n+2)+\Big[k^{\alpha-2}(2)\nonumber\\
&\ \ \ \ -\alpha k^{\alpha-2}(1)+ \frac{\alpha (\alpha -1)}{2}k^{\alpha-2}(0)\Big]S_\alpha^\lambda(n+1)\nonumber\\
&=S_\alpha^\lambda(n+3)-2S_\alpha^\lambda(n+2)+S_\alpha^\lambda(n+1)\nonumber
\end{align}
when $n\in\mathbb{N}_0$.

Thus, replacing (3.10) and (3.11) in (3.9), we deduce that
\begin{equation*}
\Delta ^2 \Delta^\alpha S_\alpha (n+1) =\Delta^2 AS_\alpha(n+1)+\gamma\Delta^2 S_\alpha^\lambda(n+1)
\end{equation*}
when $n\in\mathbb{N}_0$. The conclusion follows  easily from Lemma 3.5 and the proof is completed.
\end{proof}

\subsection{{\rm \textbf{The Existence and Uniqueness of Solution}}}
\vskip .3cm
We recall the following definition introduced in  \cite{zh}, which is necessary to  establish the main result of this section.

\begin{definition} {\rm (\cite[Lemma 3.2]{zh})}
Let $2<\alpha<3$ be given, The function $h_\alpha: \mathbb N_0 \rightarrow \mathbb R$ is defined by
$h_\alpha(0)=1,~~h_\alpha(1)=\alpha-1,~~h_\alpha(2)=\frac{\alpha(\alpha-1)}{2}$, and
 \begin{equation}
 h_\alpha(n+3)+(1-\alpha) h_\alpha(n+2)+ \frac{(\alpha-1)(\alpha-2)}{2}h_\alpha(n+1)=0
 \end{equation}
 when $n\geq0$.
 \end{definition}
We observe that the Fourier transform of $h_\alpha$ is as follows:
 \begin{equation}
\hat{h}_\alpha(z)=\frac{z^3}{z^2+(1-\alpha)z+\frac{(\alpha-1)(\alpha-2)}{2}}.
\end{equation}
by Remark 3.3 in  \cite{zh}.

\vskip .3cm

Now we are ready to prove the  main result of this section.
\begin{theorem}
Let $X$ be a Banach space, $A\in B(X)$, $2<\alpha<3$, $\gamma\in\mathbb R$, $\lambda\in\mathbb N$ and let  $f\in S(\mathbb{N}_0; X)$ be given. Then $u\in S(\mathbb{N}_0; X)$   defined by $u(0) = u(1) = u(2) = 0$, and
\begin{align}
u(n)=(h_\alpha*S_\alpha *f)(n-3)
\end{align}
when $n\geq 3$, is a solution of (1.1), and this is the unique solution of (1.1).
\end{theorem}

\begin{proof}  From Lemma 3.4, Lemma 3.6 and (3.12), we have
\begin{align}
\Delta^{\alpha}\big(h_\alpha *S_\alpha\big)(n)&=(h_\alpha*\Delta^{\alpha}S_\alpha)(n)+h_\alpha(n+3)S_\alpha(0)\\
&\ \ \ \ +h_\alpha(n+2)[S_\alpha(1)-\alpha S_\alpha(0)]+\Big[S_\alpha(2)-\alpha S_\alpha(1)\nonumber\\
&\ \ \ \ +\frac{\alpha(\alpha-1)}{2}S_\alpha(0)\Big]h_\alpha(n+1)\nonumber\\
&=(h_\alpha*\Delta^{\alpha}S_\alpha)(n)+h_\alpha(n+3)+(1-\alpha )h_\alpha(n+2)\nonumber\\
&\ \ \ \ +\frac{(\alpha-1)(\alpha-2)}{2}h_\alpha(n+1)\nonumber\\
&= A(h_\alpha* S_\alpha)(n)+\gamma (h_\alpha* S_\alpha^\lambda)(n)\nonumber
\end{align}
when $n\in \mathbb N_0$.
This, incorporating Lemma 3.4 and  Lemma 3.6,  further deduces that
\begin{align}
\Delta^{\alpha}\big(h_\alpha*S_\alpha *f\big)(n)&=(\Delta^{\alpha}(h_\alpha*S_\alpha )*f)(n)+(h_\alpha*S_\alpha )(0)f(n+3)\\
&\ \ \ \ +\big[(h_\alpha*S_\alpha )(1)-\alpha (h_\alpha*S_\alpha )(0)\big]f(n+2)\nonumber\\
&\ \ \ \ +\big[(h_\alpha*S_\alpha )(2)-\alpha (h_\alpha*S_\alpha )(1)\nonumber\\
&\ \ \ \ +\frac{\alpha(\alpha-1)}{2}(h_\alpha*S_\alpha )(0)\big]f(n+1)\nonumber\\
&=(\Delta^{\alpha}(h_\alpha*S_\alpha )*f)(n)+f(n+3)\nonumber\\
&=A(h_\alpha*S_\alpha *f)(n)+\gamma (h_\alpha*S_\alpha ^\lambda*f)(n)\nonumber\\
&\ \ \ \ +f(n+3)\nonumber
\end{align}
when $n\in \mathbb N_0$.

Notice that
\begin{align*}
(h_\alpha*S_\alpha ^\lambda*f)(n)&=\sum_{j=0}^{n}S_\alpha^\lambda(n-j)(h_\alpha*f)(j) \nonumber\\
&=\sum_{j=0}^{n-\lambda}S_\alpha(n-\lambda-j)(h_\alpha*f)(j)+\sum_{j=n-\lambda+1}^{n}S_\alpha(n-\lambda-j)(h_\alpha*f)(j)\nonumber\\
&=\sum_{j=0}^{n-\lambda}S_\alpha(n-\lambda-j)(h_\alpha*f)(j)\nonumber\\
&=(h_\alpha*S_\alpha *f)(n-\lambda)
\end{align*}
when $n\in \mathbb N_0$ as $S_\alpha(-i)= 0$ when $ i= 1, \cdots, \lambda$. This combined with (3.16) implies that
 \begin{align}
\Delta^{\alpha}\big(h_\alpha*S_\alpha *f\big)(n)&=A(h_\alpha*S_\alpha *f)(n)+\gamma (h_\alpha*S_\alpha *f)(n-\lambda)\\
&\ \ \ \ +f(n+3)\nonumber
\end{align}
when $n\in \mathbb N_0$.

Let $u: \mathbb N \to X$ be defined by $u(0) = u(1) = u(2) =0$ and $u(n) = (h_\alpha*S_\alpha   * f)(n-3)$ when $n\geq 3$.
We claim that
\begin{equation}
\Delta^\alpha u(n) = \Delta^\alpha (h_\alpha*S_\alpha   * f)(n-3)
\end{equation}
when $n\geq 3$. Indeed, since $u(0) = u(1) = u(2) =0$, we have
\begin{align*}
\Delta^{-(3-\alpha)} u(n) &= (k^{3-\alpha} * u)(n) = \sum_{j=0}^n k^{3-\alpha}(j) u(n-j)\nonumber\\
 &=\sum_{j=0}^{n-3} k^{3-\alpha}(j)(h_\alpha*S_\alpha   * f)(n-3-j)\nonumber\\
 & = \Delta ^{-(3-\alpha)}(h_\alpha*S_\alpha   * f)(n-3)
\end{align*}
when $n\geq 3$, which clearly implies that (3.18) holds as $\Delta^\alpha = \Delta^3 \Delta^{-(3-\alpha)}$ by (2.6).

 Therefore, (3.17) and (3.18) together deduce that
\begin{equation}\label{3.19}
\Delta^{\alpha}u(n)=Au(n)+\gamma u(n-\lambda)+f(n)
\end{equation}
  when $n\geq 3$ as $u(i)=0$ when $ i=-\lambda, -\lambda+1, \cdots,2$. This means that $u$  is a solution of (1.1) when $n\geq 3$.

Next we show  the equality (3.19) also holds when $0\leq n\leq 2$. Indeed, It follows from (3.14) that
\begin{align}
u(3) &= (h_\alpha*S_\alpha * f)(0) = (h_\alpha*S_\alpha) (0)f(0)\\
& = S_\alpha(0)h_\alpha(0)f(0) = f(0);\nonumber
\end{align}
\begin{align}
u(4) &= (h_\alpha*S_\alpha * f)(1) = (h_\alpha*S_\alpha) (1)f(0) + (h_\alpha*S_\alpha) (0)f(1)\\
&=[S_\alpha(0)h_\alpha(1) + S_\alpha(1)h_\alpha(0)] f(0)+ S_\alpha(0)h_\alpha(0)f(1)\nonumber\\
&=\alpha f(0) + f(1)\nonumber
\end{align}
\begin{align}
u(5) &= (h_\alpha*S_\alpha * f)(2) = (h_\alpha*S_\alpha) (2)f(0) \\
&\ \ \ \ + (h_\alpha*S_\alpha) (1)f(1) + (h_\alpha*S_\alpha) (0)f(2)\nonumber\\
&=[h_\alpha(2)S_\alpha(0) + h_\alpha(1)S_\alpha(1) + h_\alpha(0)S_\alpha(2)] f(0)\nonumber\\
 &\ \ \ + [h_\alpha(1)S_\alpha(0) + h_\alpha (0)S_\alpha (1)]f(1) +h_\alpha(0)S_\alpha(0)f(2)\nonumber\\
&=\frac{\alpha (\alpha +1)}{2} f(0) + \alpha f(1) + f(2).\nonumber
\end{align}
Thus, using the conditions $ u(0)=u(1)=u(2)=0$ and (3.20)-(3.22), we obtain
\begin{align}
\Delta^\alpha u(0) &= \Delta^3 \Delta^{-(3-\alpha)} u(0) = (k^{3-\alpha} *u)(3) = u(3) = f(0)\nonumber\\
\Delta^\alpha u(1) &= \Delta^3 \Delta^{-(3-\alpha)} u(1) =  (k^{3-\alpha} *u)(4) -3( k^{3-\alpha} *u)(3)\nonumber\\
& =  k^{3-\alpha}(0) u(4) +  k^{3-\alpha}(1)u(3) -3 k^{3-\alpha}(0)u(3)\nonumber\\
& = u(4) - \alpha u(3) = f(1)\nonumber
\end{align}
\begin{align}
\Delta^\alpha u(2) &= \Delta^3 \Delta^{-(3-\alpha)} u(2) =  (k^{3-\alpha} *u)(5) -3 (k^{3-\alpha} *u)(4) + 3(k^{3-\alpha}*u)(3)\nonumber\\
& =  k^{3-\alpha}(0) u(5) + k^{3-\alpha}(1)u(4) + k^{3-\alpha}(2)u(3) -3k^{3-\alpha}(0)u(4)\nonumber\\
&\ \ \ \ -3 k^{3-\alpha}(1)u(3)+ 3k^{3-\alpha}(0)u(3)\nonumber\\
& = u(5) - \alpha u(4) +\frac{\alpha(\alpha -1)}{2} u(3) = f(2).\nonumber
\end{align}
Therefore, $\Delta^{\alpha}u(n)=Au(n) +\gamma u(n-\lambda)+ f(n)$  when $0\leq n\leq 2$ as $u(i)=0$ when $ i=-\lambda, -\lambda+1, \cdots, 2$. We have shown that $u\in S(\mathbb{N}_0; X)$ given by (3.14) is a solution of (1.1).

It remains to show that the solution is unique. Clearly, we only need to show that  $0 \in S(\mathbb{N}_0; X)$ is the unique solution of the following homogeneous equation
\begin{equation}
 \left\{\begin{array}{ll}
\Delta^{\alpha}u(n)=Au(n)+\gamma u(n-\lambda), \ (n\in\mathbb{N}_0) \\
 u(i)=0~\text{ when}~ i=-\lambda, -\lambda+1, \cdots, 2.
  \end{array}
\right.
\end{equation}
Let $u\in S(\mathbb{N}_0; X)$ be a solution of equation (3.23).
We first show that $u(3) =0$. The identity $\Delta^{\alpha}u(0)=Au(0)+\gamma u(-\lambda) $ implies that $\Delta^{\alpha}u(0)= 0$. On the other hand,
\begin{align*}
\Delta^{\alpha}u(0)&=\Delta^3 \Delta^{-(3-\alpha)} u(0)\\
 &= \Delta^{-(3-\alpha)}u(3) - 3\Delta^{-(3-\alpha)}u(2) + 3\Delta^{-(3-\alpha)}u(1) + \Delta^{\alpha-3}u(0)\\
& = (k^{3-\alpha}*u)(3) - 3(k^{3-\alpha}*u)(2) + 3(k^{3-\alpha}*u)(1) + (k^{3-\alpha}*u)(0)\\
& = k^{3- \alpha}(0) u(3) = u(3)
\end{align*}
by the assumption $u(0)=u(1)=u(2) =0$. Consequently, $u(3) =0$

Assume that $u(n)= 0$ for all $3\leq n\leq k$ for some $k\geq 3$, we are going to show that $u(k+1) = 0$. Since $k -2 < k$, we have $u(k-2) = 0$ by assumption. Thus  $\Delta ^\alpha u(k-2) = Au(k-2)+\gamma u(k-2-\lambda) = 0$. On the other hand,
\begin{align*}
\Delta ^\alpha u(k-2) & =\Delta^3 \Delta^{-(3-\alpha)} u(k-2)\\
&= \Delta^{-(3-\alpha)}u(k+1) - 3\Delta^{-(3-\alpha)}u(k) + 3\Delta^{-(3-\alpha)}u(k-1) + \Delta^{-(3-\alpha)}u(k-2)\\
& = k^{3-\alpha}*u(k+1) - 3k^{3-\alpha}*u(k) + 3k^{3-\alpha}*u(k-1) + k^{3-\alpha}*u(k-2)\\
& = k^{3- \alpha}(0) u(k+1) = u(k+1).
\end{align*}
 by assumption $u(k-2) = u(k-1) = u(k) =0$.  Consequently, $u(k+1) = 0$ when $3\leq n\leq k$. Therefore, $u(n) =0$ for all $n\in\mathbb{N}_0$. This completes the proof.
   \end{proof}

\section{A Characterization of the $\ell^p$-Maximal Regularity}

Let $A\in B(X)$ and $f\in S(\mathbb{N}_0; X)$, where $X$ is a Banach space.  In this section, we study the $\ell^p$-maximal regularity for the fractional difference  equation with delay $\lambda\in\mathbb N$:
\begin{equation}
 \left\{\begin{array}{ll}
 &\Delta^{\alpha}u(n)=Au(n)+\gamma u(n-\lambda)+f(n),\ (n\in \mathbb N_0) \\
 &u(i)=0~\text{when}~ i=-\lambda, -\lambda+1,\cdots, 2,
   \end{array}
\right.
\end{equation}
where $1< p <\infty$, $2<\alpha < 3$ and $\gamma\in\mathbb R$.

Let $f\in S(\mathbb{N}_0; X)$ be given. By Theorem 3.8, the unique solution of (4.1) can be represented by $u(0) = u(1) = u(2) =0$ and
\begin{equation}
u(n)= (h_\alpha*S_\alpha *f\big)(n-3)
\end{equation}
when $n\geq 3$.
This means that $\Delta ^\alpha u(0)=f(0)$, $\Delta ^\alpha u(1)=f(1)$, $\Delta ^\alpha u(2)=f(2)$ and
\begin{align}
\Delta^{\alpha}u(n)&=A(h_\alpha*S_\alpha *f\big)(n-3)+\gamma (h_\alpha*S_\alpha *f\big)(n-3-\lambda)+f(n)\\
&=A(h_\alpha*S_\alpha *f\big)(n-3)+\gamma (h_\alpha*S_\alpha ^\lambda*f\big)(n-3)+f(n)\nonumber
\end{align}
when $n\geq 3$ as $u(i)=0$ when $i=-\lambda, -\lambda+1, \cdots,2$.
\vskip .3cm
In analogy to cases  $\alpha=1$ and $\alpha=2$ (see for instance \cite{bl}),  we introduce the following definition concerning maximal regularity.
\begin{definition}
 Let $1< p <\infty$, $2<\alpha<3$, $\lambda\in\mathbb N$ and let $A\in B(X)$ be given. We say  that (4.1) has the $\ell^p$-maximal regularity if
\begin{equation}
(\mathcal{E}_\alpha f)(n):=A(h_\alpha*S_\alpha   * f) (n) =A\sum_{j=0}^{n}(h_\alpha*S_\alpha )(n-j)f(j)\nonumber
\end{equation}
\begin{equation}
(\mathcal{F}_\alpha f)(n):=(h_\alpha * S_\alpha^\lambda * f) (n) =\sum_{j=0}^{n}(h_\alpha*S_\alpha ^\lambda)(n-j)f(j)\nonumber
\end{equation}
when $n\in \mathbb N_0$, define  two bounded linear operators $\mathcal{E}_\alpha, \mathcal{F}_\alpha$ on  $\ell^p(\mathbb N_0; X)$.
\end{definition}

It is easy to see that when (4.1) has the $\ell^p$-maximal regularity, then for all $f\in \ell^p(\mathbb{N}_0; X)$, the unique solution $u$ of (4.1) given by (4.2) satisfies $\Delta^{\alpha}u\in \ell^p(\mathbb N_0; X)$ by (4.3).
\vskip .3cm
We will need  the following hypothesis:
$$(\mathcal{C}_\alpha):~~~~ \sup_{n\in\mathbb N_0}\big\|S_\alpha(n)\big\|<\infty~~{\rm and}~~z^{3-\alpha}(z-1)^\alpha-\gamma z^{-\lambda}\in\rho(A)~~{\rm for~~all}~~|z|=1, \ z\not= \pm 1.$$
Now we are ready to state  the main result of this section.

\begin{theorem}
Let $X$ be a UMD space, $2<\alpha<3$, $\gamma\in\mathbb R$, $\lambda\in\mathbb N$ and let $1 < p < \infty$. Assume that $A\in B(X)$ and the assumption $(\mathcal{C}_\alpha)$ holds. Then the following statements are equivalent:
\begin{enumerate}
\item[$(i)$] $(4.1)$ has the $\ell^p$-maximal regularity;
\item[$(ii)$] the sets
\begin{equation}
\big\{e^{(3-\alpha)it}(e^{it}-1)^\alpha\big[e^{(3-\alpha)it}(e^{it}-1)^\alpha-\gamma e^{-\lambda it} -A\big]^{-1}:\ t\in \mathbb{T}\big\}
\end{equation}
and
\begin{equation}
\big\{e^{-\lambda it}[e^{(3-\alpha)it}(e^{it}-1)^\alpha-\gamma e^{-\lambda it} -A\big]^{-1}:\ t\in \mathbb{T}\big\}
\end{equation}
\end{enumerate}
are $R$-bounded.
\end{theorem}

\begin{proof}
We first prove that the implication (ii) $\Rightarrow$ (i) is true. Let
  \begin{align*}
    g_\alpha(t)&=e^{3it}(1-e^{-it})^\alpha,\ G_1(t)=g_\alpha(t)\big[g_\alpha(t)-\gamma e^{-\lambda it}-A\big]^{-1},\\
   & \ G_2(t)=e^{-\lambda it}\big[g_\alpha(t)-\gamma e^{-\lambda it}-A\big]^{-1}
  \end{align*}
when $t\in \mathbb{T}$. Then
\begin{align*}
&g_\alpha'(t)=3ig_\alpha(t)+\frac{\alpha ig_\alpha(t)}{e^{it}-1}=(3i+\frac{\alpha i}{e^{it}-1})g_\alpha(t)\\
&G_1'(t)= (3i+\frac{\alpha i}{e^{it}-1})\Big(G_1(t)-G_1^2(t)\Big)-\gamma\lambda i G_1(t)G_2(t)\\
&G_2'(t)=-\gamma iG_2(t)-(3i+\frac{\alpha i}{e^{it}-1})G_1(t)G_2(t)-\gamma\lambda iG_2^2(t)
\end{align*}
when $t\in \mathbb{T}$.
Thus the sets
\begin{equation*}
 \big\{(e^{it}-1)(e^{it}+1)G_1'(t): \ t\in \mathbb{T}\big\}
\end{equation*}
and
\begin{equation*}
\big\{(e^{it}-1)(e^{it}+1)G_2'(t): \ t\in \mathbb{T}\big\}
\end{equation*}
are $R$-bounded by assumption and Remark 2.3. By Theorem 2.4,
 there exist
 $T_\alpha, \ L_\alpha\in B(\ell^p(\mathbb Z;X))$ such that
\begin{equation}
\widehat{T_\alpha f}(e^{it})=G_1(t)\hat{f}(e^{it})
\end{equation}
and
\begin{equation}
\widehat{L_\alpha f}(e^{it})=G_2(t)\hat{f}(e^{it})
\end{equation}
when $t\in \mathbb{T}$ for all $f\in \ell^p(\mathbb Z;X)$ with finite support.
 From (4.6), (4.7) and the trivial identity
\begin{equation*}
A\big[g_\alpha(t)-\gamma e^{-\lambda it}-A\big]^{-1}= (g_\alpha(t)-\gamma e^{-\lambda it})\big[g_\alpha(t)-\gamma e^{-\lambda it}-A\big]^{-1}-I
\end{equation*}
 we deduce that
\begin{align}
A\big[g_\alpha(t)-\gamma e^{-\lambda it}-A\big]^{-1}\hat{f}(e^{it})&= (g_\alpha(t)-\gamma e^{-\lambda it})\big[g_\alpha(t)\\
&\ \ \ \ -\gamma e^{-\lambda it}-A\big]^{-1}\hat{f}(e^{it})-\hat{f}(e^{it})\nonumber
\end{align}
 defines a bounded linear operator $K_\alpha\in B(\ell^p(\mathbb Z;X))$ given by $K_\alpha f(n)=T_\alpha f(n)-\gamma L_\alpha f(n)-f(n)$ for all $f\in \ell^p(\mathbb Z;X)$ with finite support.

 It follows from Remark 3.3 and (3.13) that
\begin{align}
&\widehat {h_\alpha*S_\alpha  }(e^{it}) = e^{3it}\big[g_\alpha(t)-\gamma e^{-\lambda it}-A\big]^{-1}, \\
&\widehat {S_\alpha^\lambda * h_\alpha}(e^{it}) = e^{-\lambda it}\big[g_\alpha(t)-\gamma e^{-\lambda it}-A\big]^{-1}\nonumber
\end{align}
when $t\in \mathbb{T}$. These together with (4.8) imply that
\begin{equation}
\widehat{(\mathcal{E}_\alpha f)}(e^{it}) = e^{3it}A\big[g_\alpha(t)-\gamma e^{-\lambda it}-A\big]^{-1}\hat f(e^{it})
\end{equation}
and
\begin{equation}
\widehat{(\mathcal{F}_\alpha f)}(e^{it}) = e^{-\lambda it}\big[g_\alpha(t)-\gamma e^{-\lambda it}-A\big]^{-1}\hat f(e^{it})=\widehat{L_\alpha f}(e^{it})
\end{equation}
when $t\in \mathbb{T}$ for all $f\in \ell^p(\mathbb N_0;X)$ with finite support. Hence $\mathcal{E}_\alpha $ and $\mathcal{F}_\alpha $ are bounded linear operators on $\ell^p(\mathbb{N}_0; X)$ by (4.7) and (4.8), respectively. Here we have used the fact that the set of all $f\in \ell^p(\mathbb N_0;X)$ with finite support is dense in $\ell^p(\mathbb N_0;X)$. We have shown that
(4.1) has $\ell^p$-maximal regularity. Hence the implication (ii) $\Rightarrow$ (i) is valid.

Now assume that (i) holds true. Then
\begin{equation} \nonumber
(\mathcal{E}_\alpha f)(n)=\left\{\begin{aligned}
 &A(h_\alpha*S_\alpha *f)(n), \ (n\geq3);\\
  &0,\ \ \ \ \ \text{otherwise}.
 \end{aligned}
\right.
\end{equation}
and
\begin{equation} \nonumber
(\mathcal{F}_\alpha f)(n)=\left\{\begin{aligned}
 &(h_\alpha*S_\alpha ^\lambda*f)(n), \ (n\geq3); \\
  &0,\ \ \ \ \ \ \text{otherwise}.
 \end{aligned}
\right.
\end{equation}
define two bounded linear operators $\mathcal{E}_\alpha , \mathcal{F}_\alpha \in B(\ell^p(\mathbb N_0; X))$. It follows from (4.10)-(4.11) that
\begin{align}
& \widehat{(\mathcal{E}_\alpha f)}(e^{it})= e^{3it}T\big[g_\alpha(t)-\gamma e^{-\lambda it}-A\big]^{-1}\hat f(e^{it}),\nonumber\\
&\widehat{(\mathcal{F}_\alpha f)}(e^{it}) = e^{-\lambda it}\big[g_\alpha(t)-\gamma e^{-\lambda it}-A\big]^{-1}\hat f(e^{it})\nonumber
\end{align}
when $t\in \mathbb{T}$ for all $f\in \ell^p(\mathbb N_0;X)$ with finite support. Since the convolution operators $\mathcal{E}_\alpha $ and $\mathcal{F}_\alpha $ are translation invariant on $\ell^p(\mathbb{Z}; X)$, $\mathcal{E}_\alpha $ and $\mathcal{F}_\alpha $  can extend to  bounded linear operators on $\ell^p(\mathbb Z;X)$.  It follows from (2.9) that the functions $t\to e^{3it}A\big[g_\alpha(t)-\gamma e^{-\lambda it}-A\big]^{-1}$ and $t\to  e^{-\lambda it}\big[g_\alpha(t)-\gamma e^{-\lambda it}-A\big]^{-1}$ are $\ell^p$-Fourier multiplier.

Thus the sets
\begin{equation*}
\big\{e^{3it}A\big[g_\alpha(t)-\gamma e^{-\lambda it}-A\big]^{-1}:\ t\in \mathbb{A}\big\}~{\rm and}~\big\{e^{-\lambda it}\big[g_\alpha(t)-\gamma e^{-\lambda it}-A\big]^{-1}:\ t\in \mathbb{T}\big\}
\end{equation*}
 are $R$-bounded by Theorem 2.5.

This combined with Remark 2.3 and the trivial equality
\begin{equation*}
g_\alpha(t)\big[g_\alpha(t)-\gamma e^{-\lambda it}-A\big]^{-1}=T\big[g_\alpha(t)-\gamma e^{-\lambda it}-A\big]^{-1} +\gamma e^{-\lambda it}\big[g_\alpha(t)-\gamma e^{-\lambda it}-A\big]^{-1}+ I
\end{equation*}
 implies that the set
\begin{equation*}
\big\{g_\alpha(t)\big[g_\alpha(t)-\gamma e^{-\lambda it}-A\big]^{-1}:\ t\in \mathbb{T}\big\}
\end{equation*}
is $R$-bounded. This completes the proof.
\end{proof}

Since the second condition in Theorem 4.2 does not depend on the parameter $1 < p < \infty$, we have the following immediate consequence.

\begin{corollary}
Let $X$ be a UMD space, $2<\alpha<3$, $\gamma\in\mathbb R$ and $\lambda\in\mathbb N$. Assume that $A\in B(X)$  and  the assumption $(\mathcal{C}_\alpha)$ holds. If $(4.1)$ has the $\ell^p$-maximal regularity for some $1 < p < \infty$, then it has the $\ell^p$-maximal regularity for all $1 < p < \infty$.
\end{corollary}

If the underlying Banach space $X$ is a  Hilbert space, then the $R$-boundedness coincides with the norm boundedness \cite{ar}. This together with Theorem 4.2 give the following result.

\begin{corollary}
Let $H$ be a Hilbert space, $\gamma\in\mathbb R$, $\lambda\in\mathbb N$ and let $1 < p < \infty$. Assume that $A\in B(X)$  and the assumption $(\mathcal{C}_\alpha)$ holds. Then $(4.1)$ has the $\ell^p$-maximal regularity if and only if there exists $C > 0$  such that
\begin{equation}
\sup_{t\in \mathbb{T}}\big\Vert e^{(3-\alpha)it}(e^{it}-1)^\alpha\big[e^{(3-\alpha)it}(e^{it}-1)^\alpha-\gamma e^{-\lambda it}-A\big]^{-1}\big\Vert < \infty\nonumber
\end{equation}
\begin{equation}
\sup_{t\in \mathbb{T}}\big\Vert e^{-\lambda it}[e^{(3-\alpha)it}(e^{it}-1)^\alpha-\gamma e^{-\lambda it} -A\big]^{-1}\big\Vert < \infty\nonumber
\end{equation}
\end{corollary}

In the following example, served as an application of  Theorem 4.1, we  provide a criterion that ensures the $\ell^p$-maximal regularity of equation $(4.1)$  under a suitable assumption on the operator $A$. We observe that this example was studied by many authors \cite{li5, le} in the case  $0<\alpha < 2$. The argument used in \cite{li5, le} works as well in the case $2 < \alpha < 3$.
\begin{example}
Let $1 < p < \infty$, $2<\alpha<3$, $\gamma\in\mathbb R$, $\lambda\in\mathbb N$. Let $H$ be a Hilbert space and $A\in B(H)$ satisfying the following condition
 \begin{equation*}
 (\mathfrak{C})\ \ \ \ \|A\|<\omega_f:=\min_{t\in \mathbb{T}}|f_{\alpha,\gamma,\lambda}(t)|<1
  \end{equation*}
   where $f_{\alpha,\gamma,\lambda}(t)=e^{(3-\alpha)it}(e^{it}-1)^\alpha -\gamma e^{-\lambda it}$.
It follows form the condition $\mathfrak{C}$ and Theorem 7.3-4 in \cite{kr}  that $f_{\alpha,\gamma,\lambda}\in \rho(A)$,
\begin{equation*}
(f_{\alpha,\gamma,\lambda}-A)^{-1}=\sum_{j=0}^{\infty}\frac{A^j}{\Big(f_{\alpha,\gamma,\lambda}(t)\Big)^{j+1}}~~~{\rm and}~~~ \|(f_{\alpha,\gamma,\lambda}-A)^{-1}\|\leq\frac{1}{\omega_f-\|A\|}
 \end{equation*}
when  $t\in \mathbb{T}$.

Again using condition $\mathfrak{C}$, there exists a circle $\Gamma$ centered at the origin of the complex plane with  radius $r<1$ such that
\begin{equation} \nonumber
S_\alpha (n)=\left\{\begin{aligned}
 &\frac{1}{2\pi i}\int_{\|z\|=r}z^n\big[z^2+(1-\alpha)z+\frac{(\alpha-1)(\alpha-2)}{2}\big]\big[z^{3-\alpha}(z-1)^\alpha\\
 &\ \ \ \ \ -\gamma z^{-\lambda}-A\big]^{-1}dz, \ (n\geq3) \\
 &I, ~\text{when}~ n=0, 1, 2, \\
  &0,~\text{when}~ n=-\lambda, -\lambda+1,\cdots, -1.
 \end{aligned}
\right.
\end{equation}
It follows that
\begin{equation*}
\|S_\alpha (n)\|<\frac{4}{\omega_f-\|A\|}
 \end{equation*}
 when $n\in\mathbb{N}_0$.
It is easy to verify that the sets
\begin{equation}
\big\{e^{(3-\alpha)it}(e^{it}-1)^\alpha\big[e^{(3-\alpha)it}(e^{it}-1)^\alpha-\gamma e^{-\lambda it} -A\big]^{-1}:\ t\in \mathbb{T}\big\}\nonumber
\end{equation}
and
\begin{equation}
\big\{e^{-\lambda it}[e^{(3-\alpha)it}(e^{it}-1)^\alpha-\gamma e^{-\lambda it} -A\big]^{-1}:\ t\in \mathbb{T}\big\}\nonumber
\end{equation}
are bounded. Hence (4.1) has the $\ell^p$-maximal regularity by Corollary 4.4.
 \end{example}
\vskip .3cm

\bibliographystyle{amsalpha}

\end{document}